\documentclass[10pt]{amsart}
\usepackage{amsmath}
\usepackage{paralist}
\usepackage{epsfig} %For pictures: screened artwork should be set up with an 85 or 100 line screen
\usepackage{epstopdf} %This is to transfer .eps figure to .pdf figure; please compile your article using PDFLaTex or PDFTeXify.
\usepackage[colorlinks=true,hypertexnames=false]{hyperref}
\hypersetup{urlcolor=blue, citecolor=red}
\allowdisplaybreaks

\usepackage{tikz}
\usepackage{mathrsfs}
\usepackage{array}
\usepackage{amsfonts}
\usepackage{amsthm}
\usepackage{amssymb}
\usepackage{graphicx}
\usepackage{esint}
\usepackage{enumitem}
\usepackage[justification=centering]{caption}
\usepackage{float}

\usepackage{tikz}
\tikzset{
	treenode/.style = {align=center, inner sep=0pt, text centered,
		font=\sffamily},
	arn_r/.style = {treenode, circle, blue, draw=blue,fill=yellow, 
		text width=1em, very thick},% arbre rouge noir, noeud rouge
	dot/.style={circle,draw,inner sep=1.2,fill=black},
}
\usepackage{forest}

\usepackage{mathtools}

 \mathtoolsset{showonlyrefs} %

\usetikzlibrary{shadings}
\usepackage{color}
\usepackage{hyperref}
\hypersetup{
	colorlinks   = true, %Colours links instead of ugly boxes
	urlcolor     = blue, %Colour for external hyperlinks
	linkcolor    = blue, %Colour of internal links
	citecolor   = red %Colour of citations
}
\newcommand{\T}{\mathbb{T}}

\newcommand{\ol}{\underline{\textrm{O}}}
\newcommand{\ou}{\overline{\textrm{O}}}

\newcommand{\R}{\mathbb{R}}

\newcommand{\s}[3]{\sum_{#1 \in S^{#3} (#2)}}
\newcommand{\bt}{\beta}
\newcommand{\Sh}[1]{\mathcal{S}_{h}(#1)}
\newcommand{\cb}{c_{\beta}}
\newcommand{\itx}[1]{\fint_{\partial\T^{#1}}f(s){\rm d}s}

\definecolor{darkblue}{rgb}{0.05, .05, .9}
\definecolor{darkgreen}{rgb}{0.1, .65, .1}
\definecolor{darkred}{rgb}{0.8,0,0}

\newcounter{dummy}
\makeatletter
\newcommand\myitem[1][]{\item[(#1)]\refstepcounter{dummy}\def\@currentlabel{#1}} % to create items with special labels: Assumpstions A1,A2..
\makeatother

\makeatletter
\newcommand{\labeltext}[3][]{ % to create custom labels
	\@bsphack
	\csname phantomsection\endcsname
	\def\tst{#1}
	\def\labelmarkup{\emph}
	\def\refmarkup{\normalfont}
	\ifx\tst\empty\def\@currentlabel{\refmarkup{#2}}{\label{#3}}
	\else\def\@currentlabel{\refmarkup{#1}}{\label{#3}}\fi
	\@esphack
	\labelmarkup{#2}
}
\makeatother

\newtheorem{theorem}{Theorem}[section]
\newtheorem{corollary}{Corollary}

\newtheorem{lemma}[theorem]{Lemma}
\newtheorem{proposition}[theorem]{Proposition}

\theoremstyle{definition}
\newtheorem{definition}{Definition}
\newtheorem{remark}{Remark}

\begin{document}
	
	\title[The two membranes problem in a regular tree]{\bf The two membranes problem in a regular tree}

	\author[I. Gonz\' alvez, A. Miranda
		 and J. D. Rossi]{Irene Gonz\' alvez, Alfredo Miranda
		 and Julio D. Rossi}
		 
		 \address{ 
		 I. Gonz\'alvez. Departamento de Matem\'{a}ticas, Universidad Aut\'{o}noma de Madrid,
		Campus de Cantoblanco, 28049 Madrid, Spain.
		\newline
		\texttt{~irene.gonzalvez@uam.es}; 
		\bigskip
		\newline
		 \indent A. Miranda and J. D. Rossi. Departamento de Matem\'{a}ticas, FCEyN, Universidad  de Buenos Aires, Pabell\'{o}n I, Ciudad Universitaria (1428), Buenos Aires, Argentina.  \newline 
		 \texttt{~amiranda@dm.uba.ar,~jrossi@dm.uba.ar }		
}

	\date{}

	\maketitle

	\begin{abstract}
		In this paper we study the two membranes problem for
		 operators given in terms of a mean value formula on a regular tree.
		We show existence of solutions under adequate conditions
		on the boundary data and the involved source terms. We also show that,
		when the boundary data are strictly separated, the coincidence set is separated from the boundary and thus it contains only a
		finite number of nodes.
	\end{abstract}

\section{Introduction.}

One of the systems that attracted the attention 
of the partial differential equations community is the two membranes problem, that models
the behaviour of two elastic membranes that are clamped at the boundary of a prescribed domain, they are assumed to be ordered (one membrane is 
assumed to be above the other)
and they are subject to different external forces. 
The main assumption here is that the two membranes do not penetrate each other
(they are assumed to be ordered in the whole domain). This situation can be
modeled by two obstacle problems; 
 the lower membrane acts as an obstacle from below for the free elastic equation
 that describes the location of the upper membrane, while, conversely, the upper 
membrane is an obstacle from above for the equation for the lower membrane. 
The mathematical formulation is as follows:
given two differential operators $F(x,u,\nabla u, D^2u)$ and $G
(x,v,\nabla v, D^2v)$ find a pair of functions $(u,v)$ defined inside a prescribed domian $\Omega \subset \mathbb{R}^N$ such that
\begin{equation}
	\label{ED12FG}
	\left\lbrace 
	\begin{array}{ll}
		\displaystyle \min\Big\{ F(x,u(x) ,\nabla u(x), D^2u(x)),(u-v)(x)\Big\}=0, \quad & x\in\Omega, \\[10pt]
	\displaystyle \max\Big\{ G(x,v(x) ,\nabla v(x), D^2v(x)),(v-u)(x)\Big\}=0, \quad & x\in\Omega, \\[10pt]
	u(x)=f(x), \quad & x\in\partial\Omega, \\[10pt]
	v(x)=g(x), \quad & x\in\partial\Omega.
	\end{array}
	\right.
	\end{equation}
	
The two membranes problem for the Laplacian with a right hand side, that is, for $F(D^2u)=-\Delta u +h_1$ and $G(D^2v)=-\Delta v-h_2$, was first considered in \cite{VC} 
using variational arguments. 
When the equations that model the two membranes have a variational structure
this problem can be tackled using calculus of variations (one aims to minimize
the sum of the two energies subject to the constraint that the functions that describe the position of the membranes are always ordered inside the domain), see \cite{VC}. However, when the involved equations are not variational the analysis
relies on monotonicity arguments (using the maximum principle). 
Once existence of a solution (in an appropriate sense) is obtained a lot of interesting 
questions arise, like uniqueness, regularity of the involved functions, a description of the contact set, the regularity of the contact set, etc. See \cite{Caffa1,Caffa2,S}, the 
dissertation \cite{Vivas}
and references therein.
We also mention that a more general version of the two membranes problem
involving more than two membranes was considered by several authors (see for example  \cite{ARS,CChV,ChV}).

\medskip

Our main goal here is to introduce and analyze the obstacle problem and the two membranes problem
when the ambient space is an infinite graph with a regular structure (a regular tree) and the involved operators are given by 
mean value formulas.

\medskip

{\bf A regular tree.} Let us first describe the ambient space.
Regular trees can be viewed as discrete models of the unit ball of $\mathbb{R}^N$. 
A tree is, informally, an infinite graph (that we will denote by $\T$) in which each node but one (the root of the tree denoted by $\emptyset$) has exactly $m+1$ connected nodes, $m$
successors and one predecessor (the root has only $m$
successors). 
%\textcolor{blue}{The elements $x$ in $\T$ are called vertices and have associated a level, that we denote by $|x|$, according to the distance (in nodes) to the root. The root has zero level, its successors has level one, etc. }
A element $x$ in $\T$ is called  a vertice and it is representated as a $k$-tupla for some $k\in\mathbb{N}$ of natural numbers between $0$ and $m$, that is, 
 $x=(0,a_{1},...,a_{k}) $ where $a_{i}\in\{0,...,m-1\}$ for all $1\leq j\leq k$. Here $k$ is the level of $x$, that we denote by $|x|$, according to the distance (in nodes) to the root. The root has zero level, its successors has level one, etc. 
Each vertex $x$ has $m$ successors, that we will denote by $S^1(x)$, and are described by 
	\[
	S^1(x)=\Big\{(0,a_{1},...,a_{k},i)\colon i\in\{0,1,\dots,m-1\}\Big\}
	\]
	 where $x=(0,a_{1},...,a_{k})$. Notice that we used a digit $i\in\{0,1,\dots,m-1\}$ to enumerate the successors of $x$.
In general, we will denote
	\[
	S^l(x)=\Big\{(0,a_{1}, ...,a_{k},i_1,\dots, i_l)\colon i_j \in\{0,1,\dots,m-1\} \quad \mbox{for} \quad 1\leq j\leq l\Big\}
	\]
the set of successors of  $x=(0,a_{1},...,a_{k})$ with $|x|+l$ level. By convention, \mbox{$S^0(x)=\{x\}$.}
If $x$ is not the root then $x$ has a unique 
immediate predecessor, which we will denote $\hat{x}.$
We use the notation $x^j$ to denote the predecessor of $x$ with $|x^j|=j$, that is, for a node $x$ with $|x|=k$, we have $x=(0,a_1,\dots,a_j,\dots,a_k)$, then $x^j=(0,a_1,\dots,a_j)$.   
Let $\T^x\subset\T$ be the subtree with root $x$ (the subset of the tree composed
by all the consecutive successors of $x$). A branch of $\T$ (that we denote by $z$) is an infinite sequence of vertices starting at the root, where each of the vertices in the sequence 
is followed 
by one of its immediate successors.
The collection of all branches forms the boundary of $\T$, denoted 
by $\partial\T$.
Observe that the mapping $\psi:\partial\T\to[0,1]$ defined as
\[
\psi(z)=\sum_{k=0}^{+\infty} \frac{a_k}{m^{k}}
\]
is surjective, where $z=(0,a_1,\dots, a_k,\dots)\in\partial\T$ and
$a_k\in\{0,1,\dots,m-1\}$ for all $k\in\mathbb{N}.$ When
$x=(0,a_1,\dots,a_k)$ is a vertex, we set$
\psi(x)=\psi((0,a_1,\dots,a_k,0,\dots)).$

\medskip

{\bf Mean value formulas and operators.}
Given a parameter $0\leq \beta\leq 1$ we define
the averaging operator $L$ acting on functions $u: \T \mapsto \mathbb{R}$ as follows:
\begin{equation}
	\label{L1}L (u)(x)=u(x) - \beta u(\hat{x})-(1-\beta)\Big(\frac{1}{m}\sum_{y\in S^1 (x)} u(y)\Big),\;\qquad x\not =\emptyset
\end{equation}
and 
\begin{equation}
\label{L2}L (u)(\emptyset)=u(\emptyset) - \Big(\frac{1}{m}\sum_{y\in S^1 (\emptyset)} u(y)\Big).
\end{equation}
The operator $L$ acting on $u$ at a vertex $x\in \T$ is given by
the difference between 
the value of $u$ at that node and the mean value
of $u$ at the vertices that are connected with $x$ (with a weight given by
$\beta$ for the predecessor and $(1-\beta)/m$ for the successors).
Note that here the distinction in the definition of $Lu$ based on whether $x=\emptyset$ or not relies on the fact that the root has no predecessor.

Now, given a function $h : \T \mapsto \mathbb{R}$, a solution to 
$$L (u) (x) = h(x) $$ is a function $u: \T \mapsto \mathbb{R}$ that verifies
$$
u(x)=\beta u(\hat{x})+(1-\beta)\Big(\frac{1}{m}\sum_{y\in S^1 (x)} u(y)\Big)+h(x),\qquad\;x\not =\emptyset,
$$
and the equation
at the root of the tree that only involves the successors and is given by
$$
u(\emptyset) = \Big(\frac{1}{m}\sum_{y\in S^1 (\emptyset)} u(y)\Big)+h(\emptyset).
$$
Notice that the equation $L (u) (x) = h(x) $ is an analogous in the tree
of the partial differential equation $-\Delta u (x) = h(x)$ in the unit ball
of $\mathbb{R}^N$ (recall the mean value formula for harmonic functions).

Next, in order to impose boundary conditions, we want to make precise what we understand for $u|_{\partial \T} = f$.

\begin{definition}
	Given $f:[0,1]\rightarrow \R$ we say that $u$ verify the boundary condition $u|_{\partial \T} = f$ if 
	\begin{equation}
		\lim_{x\rightarrow z\in\partial\T}u(x)= f(\psi(z)),
	\end{equation}
	that is, if $z=(0, a_1,\dots, a_k,\dots)\in\partial\T$, then
	\begin{equation}
		\lim_{n\to \infty}u((0,a_1,...,a_n))= f
		\Big(\sum_{k=1}^{+\infty} \frac{a_k}{m^{k}}\Big).
	\end{equation}
Here and in what follows we will take this limit uniformly, that is, given $\varepsilon >0$ there exists $K\in\mathbb{N}$ such that 
$|u(x) - f(\psi(z))| < \varepsilon$ for every $|x|\geq K$ and every $z\in\partial\T^x$.
\end{definition}

Note that, if $u|_{\partial \T} = f$, due that the limit is uniform, we have that 
	\begin{equation}
	\label{1}	\lim\limits_{N\to\infty}\frac{1}{m^N}\sum_{y\in S^{N}(x)}u(y)=\itx{x}
	\end{equation}

Now, let us state the precise definition of being a 
subsolution/supersolution to $L u=h$ in $\T$, with a boundary condition $u|_{\partial \T} = f$.

\begin{definition}
	Given an operator $L$ defined as before, a function $h : \T \rightarrow \mathbb{R}$ and  a continuous function $f:[0,1]\rightarrow\R$, we say that $u:\T\rightarrow \R$ is a subsolution (resp. supersolution)
	to $Lu=h$ in $\T$ with $u|_{\partial \T} = f$ if it verifies
	\begin{equation}
		\left\lbrace
		\begin{array}{ll}
			\displaystyle L (u)(x) \leq h(x) \quad (resp \geq), & x\in \T,\\[10pt]
			\displaystyle \limsup_{x\rightarrow z\in\partial\T}u(x)\leq f(\psi(z)) \quad (resp \liminf \geq).
		\end{array}
	\right.
	\end{equation}
We say that $u$ is a solution if it is both a subsolution and a supersolution.
\end{definition}

With this definition at hand we can introduce the obstacle problem
(see also Section \ref{sect-soluciones}). 

\begin{definition} \label{def.obs.intro}
	Given an operator $L$, a function $h : \T \rightarrow\mathbb{R}$, a continuous boundary datum $f:[0,1]\rightarrow\R$ and a function $\varphi:\T\rightarrow\mathbb{R}$ (the obstacle), we say that $u:\T\rightarrow \R$ is the solution 
	to the obstacle problem for $L-h$ in $\T$ with $u|_{\partial \T} \geq f$ and
	obstacle $\varphi$ from below ($\varphi$ is assumed to satisfy 
	$\limsup_{x\rightarrow z\in\partial\T}\varphi(x)< f(\psi(z))$) if it verifies
	\begin{equation} \label{obst.u}
	\displaystyle u(x) = \inf \left\{ w(x) : 
	\begin{array}{l} L(w)(x) \geq h(x), \quad x\in \T, \\[6pt]
		\displaystyle  w(x) \geq \varphi(x), \quad x\in\T,\\[6pt]
	 \displaystyle \mbox{ and } \liminf_{x\rightarrow z\in\partial\T}w(x)\geq f(\psi(z))
	 \end{array} \right\}.
			\end{equation}
			Under adequate conditions on $h$ (see below), this is function $u(x)$ is the unique function that satisfies
			$$
			\left\{
			\begin{array}{l}
			\displaystyle
			0=
\max \Big\{ -L( u)({x})+h(x),
\varphi(x)-u(x) \Big\},\quad x\in \T, \\[10pt]
\displaystyle
 \lim_{x\rightarrow z\in\partial\T}u(x)= f(\psi(z)).
\end{array}
\right.
$$

Analogously, given an operator $L$, a function $h : \T \rightarrow \mathbb{R}$, a continuous boundary datum $g:[0,1]\rightarrow\R$ and a function $\varphi:\T\rightarrow\mathbb{R}$ (the obstacle), we say that $v:\T\rightarrow \R$ is the solution 
	to the obstacle problem for $L-h$ in $\T$ with $v|_{\partial \T} = g$ and
	obstacle $\varphi$ from above (here $u$ is assumed to satisfy 
	$\liminf_{x\rightarrow z\in\partial\T}\varphi(x)> g(\psi(z))$) if it verifies
	\begin{equation}
	\displaystyle v(x) = \sup \left\{ w(x) : 
	\begin{array}{l} L(w)(x)\leq h(x), \quad x\in \T, \\[6pt]
		\displaystyle w(x) \leq \varphi(x), \quad x\in \T, \\[6pt] 
	 \displaystyle \mbox{ and } \limsup_{x\rightarrow z\in\partial\T}w(x)\leq g(\psi(z)) 
	 \end{array} \right\}.	
			\end{equation}
			Under adequate conditions on $h$, this can be written as 
			$$
			\left\{
			\begin{array}{l}
			\displaystyle
			0=
\min \Big\{ -L (v) (x) +h(x),
\varphi(x) - v (x) \Big\},\quad x\in \T, 
\\[10pt]
\displaystyle
 \lim_{x\rightarrow z\in\partial\T}v(x)= g(\psi(z)).
\end{array}
\right.
$$
			\end{definition}
			
			With the definition of the obstacle problem we can describe the two membranes problem in $\T$.

			\begin{definition}
			Let  $L_{1}$ and $L_{2}$  be two  averaging operators defined as in \eqref{L1} and \eqref{L2}
			(with different $\beta_1$, $\beta_2$), $h_1,h_2 : \T
		\rightarrow\mathbb{R}$ two functions and $f,g:[0,1]\rightarrow\R$
			two continuous boundary 
			conditions. A pair $(u,v):\T\rightarrow\mathbb{R}$ is a solution to the two membranes problem if $u$ is 
			the solution 
	to the obstacle problem for $L_{1}-h_1$ in $\T$ with $u|_{\partial \T} = f$ and
	obstacle $v$ from below and $v$ is 
			the solution 
	to the obstacle problem for $L_{2}-h_2$ in $\T$ with $v|_{\partial \T} = g$ and
	obstacle $u$ from above.
Recall that $u$ is the solution to the obstacle problem for $L_{1}-h_1$ in $\T$ with $u|_{\partial \T} = f$ and obstacle $v$ from below if it is the infimum of supersolutions for $L_{1}-h_1$ with boundary datum $f$ that are above the obstacle $v$, and analogously, $v$ is the solution to the obstacle problem for $L_{2}-h_2$ in $\T$ with $v|_{\partial \T} = g$ and obstacle $v$ from above if it is the supremum of subsolutions for $L_{2}-h_2$ with boundary datum $g$ that are above the obstacle $u$.
Summarizing, a pair $(u,v)$ is a solution to the two membranes problem
if it solves the system 
\begin{equation}
\label{ED1.general.intro}
\left\lbrace
\begin{array}{l}
\displaystyle 0= 
\max \! \Big\{ -L_1 (u) (x) +h_1(x),
v(x)-u(x) \Big\},\ x\in \T , \\[10pt]\displaystyle
0= 
\min \!\Big\{  -L_2 (v)(x)+h_2(x),
u(x) - v(x)  \Big\},\ x\in \T, 
\end{array}
\right.
\end{equation}
with
$$
	\left\lbrace
	\begin{array}{ll}
		\displaystyle \lim_{x\rightarrow z\in\partial\T}u(x)= f(\psi(z)),\\[10pt]
		\displaystyle \lim_{x\rightarrow z\in\partial\T}v(x)= g(\psi(z)).
	\end{array}
	\right.
$$
\end{definition}

Let us point out that \eqref{ED1.general.intro} has a probabilistic interpretation that
we will describe in Section \ref{sect-probabil}. 

	Now, for short, we introduce a notation, for a node $x$, let
	$$
	\Sh{x}:= \frac{1}{\bt}\sum_{i=1}^{\infty}\sum_{j=0}^{i-1}\Big(\frac{\bt}{1-\bt}\Big)^{k-j}\s{y}{x}{j}\frac{h(y)}{m^{j}}.
	$$

With this notation at hand, we are ready to state our main result for the two membranes problem.

\begin{theorem} \label{teo.main.intro}
Given two averaging operators, $L_{1}$ and $L_{2}$ defined as in \eqref{L1} and \eqref{L2} (that involve two different
parameters $\beta_1$, $\beta_2$), two different functions $h_1,h_2: \T \rightarrow \mathbb{R}$ and two continuous boundary 
			conditions $f,g:[0,1]\rightarrow\R$ with $f > g$ in $[0,1]$, 
			such that
			\begin{equation} \label{cond.sistema.intro}
		\begin{array}{l}	
		\displaystyle 0\leq \beta_1 < \frac12, \qquad 0\leq\beta_2 < \frac12, \\[10pt] 
				\displaystyle	
				\lim_{x\in \T, |x|=k \to \infty} \sum_{j=1}^{k}
			\Big(\frac{\beta}{1-\beta}\Big)^{k-j}\! S_{h_1}(x^j) = 0,  \\[10pt] 
				\displaystyle	\lim_{x\in \T, |x|=k \to \infty} \sum_{j=1}^{k}
			\Big(\frac{\beta}{1-\beta}\Big)^{k-j}\! S_{h_2}(x^j) = 0,
\end{array}
		\end{equation}
then,	 there exists a pair $(u,v):\T\mapsto \mathbb{R}$ that is a solution to the two membranes problem. That is, $(u,v)$ solves the system \eqref{ED1.general.intro}.
			
			Moreover, under these conditions the coincidence set $\{x\in\T : u(x) = v(x) \}$ is finite.
\end{theorem}

To prove this result we first deal with a single equation and 
find necessary and sufficient conditions for the existence of a solution. 
We have the following result that we believe that has its own interest.

\begin{theorem} \label{teo.intro.una.eacuacion}
Given an averaging operator, $L$, a function $h: \T\rightarrow\mathbb{R}$ and a continuous boundary 
			condition $f:[0,1]\rightarrow\R$, there exists
			a unique bounded solution to
			\begin{equation}
	\label{EQL.intro}
	\left\lbrace
	\begin{array}{ll}
		\displaystyle L(u)(x)= h(x), \quad  & x\in \T,\\[8pt]
		\displaystyle \lim_{x\rightarrow z\in\partial\T}u(x)= f(\psi(z)).
	\end{array}
	\right.
\end{equation}
			if and only if
			\begin{equation} \label{cond.una.ecuacion.intro}
		\begin{array}{l}	
		\displaystyle 0\leq \beta < \frac12, \\[10pt] 
				\displaystyle	
				\lim_{x\in \T, |x|=k \to \infty} \sum_{j=1}^{k}
			\Big(\frac{\beta}{1-\beta}\Big)^{k-j}\! S_h(x^j) = 0.
			\end{array}
		\end{equation}	
		\end{theorem}

Under the conditions in the statement of the previous theorem we can also
construct sub and supersolutions that are the key to obtain solvability
of the obstacle problem (from above or below). 

Notice that when $\beta <\frac12$ if $h$ is such that
$$
\lim_{x\in \T, |x|=k \to \infty}  S_h(x) = 0
$$
then the second condition in \eqref{cond.una.ecuacion.intro} holds.
In fact, we have 
$$
\begin{array}{l}
\displaystyle 
\sum_{j=1}^{k-1}
			\Big(\frac{\beta}{1-\beta}\Big)^{k-j}\! S_h(x^j)
			= \sum_{j=1}^{k_0}
			\Big(\frac{\beta}{1-\beta}\Big)^{k-j}\! S_h(x^j) +
			\sum_{j=k_0}^{k-1}
			\Big(\frac{\beta}{1-\beta}\Big)^{k-j}\! S_h(x^j) \\[10pt]
			\displaystyle
			\qquad \leq C \Big(\frac{\beta}{1-\beta}\Big)^k \sum_{j=1}^{k_0}
			\Big(\frac{\beta}{1-\beta}\Big)^{-j} +
			\max_{j\geq k_0 } |S_h(x^j)| \sum_{j=k_0}^{k-1}
			\Big(\frac{\beta}{1-\beta}\Big)^{k-j}\! 
			\end{array}
$$
that is small if we first choose $k_0$ large (in order to make $\max_{j\geq k_0 } |S_h(x^j)|$
small) and then send $k$ to infinity. 

To look for examples of functions $h$ that satisfy our condition we take a function that
depends only on the level $k$, that is, $h(y) = \widehat{H} (k)$. Then, for
$0<\beta < \frac12$ the condition $S_h(x) \to 0$ reads as
 $$
\lim_{x\in\T, |x| \to \infty}  \frac{1}{\bt}\sum_{k=1}^{\infty} H(k) \sum_{j=0}^{k-1}\Big(\frac{\bt m}{1-\bt}\Big)^{k-j}  =0.
 $$
Therefore, for $H$ constant, $H=cte$, this condition is not verified.
When $\frac{\bt m}{1-\bt} < 1$ any $H(k)$ with $\lim_{k \to \infty} H(k) =0$ 
satisfies the condition, but any $H(k)>0$ with $\liminf_{k \to \infty} H(k) >0$ does not. 
On the other hand, when $\frac{\bt m}{1-\bt} \geq 1$ we need
that $H(k)$ goes to zero very fast in order to verify the condition.

With the result for a single equation at hand, to prove our main result concerning the two membranes problem, we 
use the strategy of iterate the obstacle problem from above or below. 
Starting with $v_0$ a subsolution to $L_{2}-h_2$ with boundary datum $g$, we let 
$u_1$ the solution to the obstacle problem for $L_{1}-h_1$ with boundary datum $f$
and obstacle $v_0$. Then we take $v_1$ as the solution to the obstacle problem for the second operator $L_{2}-h_2$ with boundary datum $g$
and obstacle $u_1$ and so on. 
In this way, we obtain two sequences, $u_n$, $v_n$, that we prove that are monotone and converge to a solution to the two membranes problem. To show that the boundary values are attained we need to use super and subsolutions for a single equation.
For a similar iteration procedure for second order elliptic partial differential 
equations in a bounded domain in the Euclidean space we refer to
\cite{nosotrosGMR}.

Concerning previous results in the literature for mean value formulas on trees
we refer to \cite{ary,BBGS,delpeFreRo,DPMR1,DPMR2,DPMR3,KLW,KW,MiMosRossi,
s-tree,s-tree1} and references therein. In \cite{delpeFreRo} it was studied the solvability for the single
equation with $h \equiv 0$ and proved that there is a bounded solution  
for any continuous boundary datum $f$ if and only if $0\leq \beta<\frac12$. Our results
include the case $h\neq 0$ with the condition \eqref{cond.una.ecuacion.intro}. For systems of mean value formulas
on trees we quote \cite{MiMosRossi} where two coupled equations (but not of obstacle type) were considered (a probabilistic interpretation of the equations is also provided there).

\medskip

The paper is organized as follows: in Section \ref{sect-soluciones} we analyze
conditions for existence of a solution to a single equation. Section \ref{sect-membranes} 
we deal with the two membranes problem. Finally, in Section \ref{sect-probabil},
we present a probabilistic interpretation of the two membranes problem
using game theory.

\section{Conditions for existence of bounded solutions to a single equation} 
\label{sect-soluciones}

In this section we want to find conditions on $\beta$ and the function $h:\T\rightarrow\R$ such that for every continuous function $f:[0,1] \to \mathbb{R}$
there exists a bounded solution to the Dirichlet problem 
\begin{equation}
	\label{EQL}
	\left\lbrace
	\begin{array}{ll}
		\displaystyle L(u)(x)= h(x), \quad  & x\in \T,\\[10pt]
		\displaystyle \lim_{x\rightarrow z\in\partial\T}u(x)= f(\psi(z)).
	\end{array}
	\right.
\end{equation}

\begin{proof}[Proof of Theorem \ref{teo.intro.una.eacuacion}]
First, suppose that such a bounded solution exists, then we have a function $u : \T \mapsto \mathbb{R}$ such that
	\begin{equation}
	\label{pg1}
	u(x)=\beta u(\hat{x})+(1-\beta)\Big(\frac{1}{m}\sum_{y\in S^1 (x)} u(y)\Big)+h(x),
	\end{equation}
	for $x \in\T\setminus\{\emptyset\}$ and $|x|=k$, if we write $u(x)=\beta u(x)+(1-\beta)u(x)$ we get
	\begin{equation}
	\bt \Big(u(x)-u(\hat{x})\Big)=(1-\bt)\Big(\frac{1}{m}\sum_{y\in S^1(x)} (u(y)-u(x))\Big)+h(x).
	\end{equation}
	We define $w(x)$ for $x\in\T\setminus\{\emptyset\}$ as the increment of $u$ between $x$ and its predecessor, i.e, 
	\begin{equation*}
	w(x):=u(x)-u(\hat{x}).
	\end{equation*}
	Then, the previous equation written in terms of $w$ reads as 
	\begin{align*}
	&\bt w(x)=(1-\bt)\Big(\frac{1}{m}\sum_{y\in S^1(x)} w(y)\Big)+h(x)\\
	&w(x)=\Big(\frac{1-\bt}{\bt}\Big)^{i}\frac{1}{m^{i}}\s{y}{x}{i}w(y)+\frac{1}{\bt}\sum_{j=0}^{i-1}\Big(\frac{1-\bt}{\bt}\Big)^{j}\s{y}{x}{j}\frac{h(y)}{m^{j}}\\
	&w(x)\Big(\frac{\bt}{1-\bt}\Big)^{i}=\frac{1}{m^{i}}\s{y}{x}{i}w(y)+\frac{1}{\bt}\sum_{j=0}^{i-1}\Big(\frac{\bt}{1-\bt}\Big)^{i-j}\s{y}{x}{j}\frac{h(y)}{m^{j}}.
	\end{align*}
	Adding from $i=1$ to infinity, we get
	\begin{equation}
	\label{g6} w(x)\sum_{i=1}^{\infty}\Big(\frac{\bt}{1-\bt}\Big)^{i}=\sum_{i=1}^{\infty}\s{y}{x}{i}\frac{w(y)}{m^{i}} + \frac{1}{\bt}\sum_{i=1}^{\infty}\sum_{j=0}^{i-1}\Big(\frac{\bt}{1-\bt}\Big)^{i-j}\s{y}{x}{j}\frac{h(y)}{m^{j}}.
	\end{equation}
	If $h \equiv 0$ we have existence of a solution such that
	\begin{equation}
	\label{g6.88} w(x)\sum_{i=1}^{\infty}\Big(\frac{\bt}{1-\bt}\Big)^{i}=\sum_{i=1}^{\infty}\s{y}{x}{i}\frac{w(y)}{m^{i}} 
	\end{equation}
	and since $w$ is bounded (this follows from the fact that $u$ is assumed to be bounded)
	we have that the right hand side is finite. Hence, we obtain the convergence of the series 
	$$
	\sum_{i=1}^{\infty}\Big(\frac{\bt}{1-\bt}\Big)^{i}
	$$
	that is equivalent to the condition 
	$$
	\frac{\bt}{1-\bt} <1 \iff 0<\beta<\frac12.
	$$
	
	Now, we use the notations
	\begin{align*}
	&\Sh{x}:= \frac{1}{\bt}\sum_{i=1}^{\infty}\sum_{j=0}^{i-1}\Big(\frac{\bt}{1-\bt}\Big)^{i-j}\s{y}{x}{j}\frac{h(y)}{m^{j}},\\
	&\cb:=\sum_{i=0}^{\infty}\Big(\frac{\bt}{1-\bt}\Big)^{i}=\frac{1-\bt}{1-2\bt},\;\;\textrm{ for }\beta\in \Big (0,\frac{1}{2}\Big).
	\end{align*}
	With this notation it is straightforward to check that,
	\begin{equation*}
	\cb-1=\sum_{i=1}^{\infty}\Big(\frac{\bt}{1-\bt}\Big)^{i}>0.
	\end{equation*}
	On the other hand, we claim that 
	\begin{equation}
	\label{g7}	\sum_{i=1}^{\infty}\s{y}{x}{i}\frac{w(y)}{m^{i}} =\itx{x}-u(x).
	\end{equation}
	Assuming this claim, \eqref{g6} can be rewritten as 
	\begin{equation}
	\label{g1}w(x)\Big(\cb-1\Big)=\itx{x}-u(x)+\Sh{x}.
	\end{equation}

		From this, we obtain that the solution $u:\T \mapsto
	\mathbb{R}$ verifies 
	\begin{equation} \label{ttt}
	u(x)=\frac{\cb-1}{\cb}u(\hat{x})+\frac{1}{\cb}\itx{x}+\frac{1}{\cb}\Sh{x},\;\;x\not=\emptyset.
	\end{equation}
	Now, we have 
	$$
	u(\emptyset)= \frac{1}{m}\sum_{y\in S^1 (\emptyset)} u(y)+h(\emptyset),
	$$
	and, using \eqref{ttt} for $x \in S^1 (\emptyset)$ we get
	$$
	u(\emptyset)= \frac{1}{m}\sum_{y\in S^1 (\emptyset)} 
	 \Big( \frac{\cb-1}{\cb}u(\emptyset)+\frac{1}{\cb}\itx{y}+\frac{1}{\cb}\Sh{y}\Big)+h(\emptyset),
	$$	
	  that is 
			$$
		u(\emptyset)= \frac{1}{m}\sum_{y\in S^1 (\emptyset)} 
		\Big( \itx{y}+\Sh{y}\Big)+ c_\beta h(\emptyset).
		$$

		Note that 
		\begin{equation*}
			\int_{0}^{1}f(s)ds=\frac{1}{m}\sum_{y\in S^1 (\emptyset)} 
			\itx{y}
		\end{equation*}
	 and then we obtain
$$
u(\emptyset)=\int_{0}^{1}f(s){\rm d}s+\frac{1}{m}\s{y}{\emptyset}{1}\Sh{y}+\cb h(\emptyset).
$$

	Therefore, we have that $u$
	satisfies the recurrence
		\begin{equation}
			\label{solSE}
		\left\{
		\begin{array}{l}
			\displaystyle
			u(x)=\frac{\cb-1}{\cb}u(\hat{x})+\frac{1}{\cb}\itx{x}+\frac{1}{\cb}\Sh{x},\;\;x\not=\emptyset,\\[10pt]
			\displaystyle u(\emptyset)=\int_{0}^{1}f(s){\rm d}s+\frac{1}{m}\s{y}{\emptyset}{1}\Sh{y}+\cb h(\emptyset).
		\end{array}
		\right.
	\end{equation} 
	If we iterate this recurrence up to the root we get the following formula for $u$; given $x\in\T$ with $|x|=k$ for $k\geq 1$,
	\begin{equation}
		\label{Formuladeu}
		\begin{array}{ll}
			\displaystyle u(x) \! = \! \Big(\frac{\beta}{1-\beta}\Big)^k\int_0^1 f(s) ds+\frac{1}{c_{\beta}}\sum_{j=1}^{k}\Big(\frac{\beta}{1-\beta}\Big)^{k-j}\itx{x^j}\\[10pt]
			\displaystyle\quad \qquad+\Big(\frac{\beta}{1-\beta}\Big)^k \! \Big[\frac{1}{m}\!\sum_{y\in S^1(\emptyset)}S_h(y)+c_{\beta}h(\emptyset)\Big]\!+\! \frac{1}{c_{\beta}}
			\!\sum_{j=1}^{k}
			\Big(\frac{\beta}{1-\beta}\Big)^{k-j}\! S_h(x^j)
		\end{array}	
	\end{equation}
	where $x^j$ denotes the predecessor of $x$ with $|x^j|=j$, that is, in terms of the notation given in the introduction, for a node $x$ with $|x|=k$, we have $x=(a_1,\dots,a_j,\dots,a_k)$ with $0\leq a_j\leq m-1$, then $x^j=(a_1,\dots,a_j)$. 
		
		This function $u$ given by \eqref{Formuladeu} is well-defined, finite, and solves que equation $L(u)(x)= h(x)$ for
	$ x\in \T$.
	
	Now we just take $f\equiv 0$ and we obtain that $u$, given by
	\begin{equation}
		\label{Formuladeu-f=0}
		\begin{array}{ll}
			\displaystyle u(x)  = \Big(\frac{\beta}{1-\beta}\Big)^k \! \Big[\frac{1}{m}\!\sum_{y\in S^1(\emptyset)}S_h(y)+c_{\beta}h(\emptyset)\Big]\!+\! \frac{1}{c_{\beta}}
			\!\sum_{j=1}^{k}
			\Big(\frac{\beta}{1-\beta}\Big)^{k-j}\! S_h(x^j)
		\end{array}	
	\end{equation}
	must satisfy 
	$$
\lim_{x\in\T, |x| \to \infty } u(x) =0	
	$$
	uniformly. Since $\beta < \frac12$,
	it is clear that 
	$$
	\lim_{k \to \infty } \Big(\frac{\beta}{1-\beta}\Big)^k \! \Big[\frac{1}{m}\!\sum_{y\in S^1(\emptyset)}S_h(y)+c_{\beta}h(\emptyset)\Big] =0.
	$$
	Hence, we must have 
	$$
	\lim_{k \to \infty }  \sum_{j=1}^{k}
			\Big(\frac{\beta}{1-\beta}\Big)^{k-j}\! S_h(x^j) = 0
	$$
	as we wanted to show. 
	
	\medskip

	{\bf Proof of the claim \eqref{g7}.}
	We  will prove by induction in $N$ that
	\begin{equation*}
	\lim\limits_{N\to\infty}\sum_{k=1}^{N}\s{y}{x}{k}\frac{w(y)}{m^{k}} =\itx{x}-u(x).
	\end{equation*}
	For $N=2$ we have that 
	\begin{align*}
	\sum_{k=1}^{2}& \s{y}{x}{k}\frac{w(y)}{m^{k}} =\s{y}{x}{1}\frac{w(y)}{m^{1}}+\s{y}{x}{2}\frac{w(y)}{m^{2}}\\
	&=\frac{1}{m}\Big(\s{y}{x}{1}\Big(u(y)-u(x)\Big)\Big)
 +\frac{1}{m^{2}}\Big(\s{y}{x}{1}\s{z}{y}{1}\Big(u(z)-u(y)\Big)\Big)\\
	&= \frac{1}{m}\Big(\s{y}{x}{1}u(y)\Big)-u(x) 
	+\frac{1}{m^{2}}\Big(\s{y}{x}{2}u(y)
	\Big)-\frac{1}{m}\Big(\s{y}{x}{1}u(y)\Big)\\
	&=\frac{1}{m^{2}}\Big(\s{y}{x}{2}u(y)
	\Big)-u(x).
	\end{align*}
	Now, our inductive hypothesis is
	\begin{equation*}
	\sum_{k=1}^{N}\s{y}{x}{k}\frac{w(y)}{m^{k}} =\frac{1}{m^{N}}\Big(\s{y}{x}{N}u(y)
	\Big)-u(x),
	\end{equation*}
	and we have to prove that for $N+1$, 
	\begin{align*}
	\sum_{k=1}^{N+1}& \s{y}{x}{k}\frac{w(y)}{m^{k}} =\s{y}{x}{N+1}\frac{w(y)}{m^{N+1}}+\frac{1}{m^{N}}\Big(\s{y}{x}{N}u(y)
	\Big)-u(x)\\
	&=\frac{1}{m^{N+1}}\Big(\s{y}{x}{N}\s{z}{y}{1}\Big(u(z)-u(y)\Big)\Big)+\frac{1}{m^{N}}\Big(\s{y}{x}{N}u(y)
	\Big)-u(x)\\
	&=\frac{1}{m^{N+1}}\Big(\s{y}{x}{N+1}u(y)\Big)-\frac{1}{m^{N}}\Big(\s{y}{x}{N}u(y)\Big)+\frac{1}{m^{N}}\Big(\s{y}{x}{N}u(y)
	\Big)-u(x)\\
	&=\frac{1}{m^{N+1}}\Big(\s{y}{x}{N+1}u(y)\Big)-u(x).
	\end{align*}
	Since $u$ have supposed that $u|_{\partial \pi}=f$, we have that 
	$$\lim\limits_{N\to \infty}\frac{1}{m^N}\s{y}{x}{N}u(y)=\itx{x}$$ due to the fact in the definition of $u|_{\partial \pi}=f$ the convergence is uniform.
	This ends the proof of the claim. 
	
	\medskip
	
	Conversely, assuming that the conditions on $\beta$ and $h$, \eqref{cond.una.ecuacion.intro} hold, from our previous computations, 
	\eqref{Formuladeu} gives us a way to construct solutions to our problem.
	In fact, we can define the function $u:\T \mapsto
	\mathbb{R}$ given by 
	\begin{equation}
		\label{Formuladeu.66}
		\begin{array}{ll}
			\displaystyle u(x) \! = \! \Big(\frac{\beta}{1-\beta}\Big)^k\int_0^1 f(s) ds+\frac{1}{c_{\beta}}\sum_{j=1}^{k}\Big(\frac{\beta}{1-\beta}\Big)^{k-j}\itx{x^j}\\[10pt]
			\displaystyle\quad \qquad+\Big(\frac{\beta}{1-\beta}\Big)^k \! \Big[\frac{1}{m}\!\sum_{y\in S^1(\emptyset)}S_h(y)+c_{\beta}h(\emptyset)\Big]\!+\! \frac{1}{c_{\beta}}
			\!\sum_{j=1}^{k}
			\Big(\frac{\beta}{1-\beta}\Big)^{k-j}\! S_h(x^j)
		\end{array}	
	\end{equation}
	where, as before, $x^j$ denotes the predecessor of $x$ with $|x^j|=j$. 
	
	One can check that $u$, given by the explicit expression \eqref{Formuladeu.66}, satisfies
	\begin{equation}
	 Lu(x)=h(x), \qquad x\in \T.
	 	\end{equation} 
	
	Now, let us check that, when we have the solvability conditions
	\begin{equation} \label{cond.una.ecuacion.88}
		\begin{array}{l}	
		\displaystyle 0\leq \beta < \frac12, \\[10pt] 
				\displaystyle	
				\lim_{x\in \T, |x|=k \to \infty} \sum_{j=1}^{k}
			\Big(\frac{\beta}{1-\beta}\Big)^{k-j}\! S_h(x^j) = 0,
			\end{array}
		\end{equation}
	it holds that
	$$
	\lim_{x\rightarrow z\in\partial\T}u(x)=f(\psi(z))
	$$
	uniformly. 
	
	From \eqref{Formuladeu.66} 
	we get 
	\begin{equation}
		\label{Formuladeu.88}
		\begin{array}{ll}
			\displaystyle u(x) -  f(\psi(z))
\\[10pt]
			\displaystyle\quad \qquad= \! \Big(\frac{\beta}{1-\beta}\Big)^k\int_0^1 f(s) ds+\frac{1}{c_{\beta}}\sum_{j=1}^{k}\Big(\frac{\beta}{1-\beta}\Big)^{k-j}\itx{x^j}- f(\psi(z))
\\[10pt]
			\displaystyle\qquad \qquad+\Big(\frac{\beta}{1-\beta}\Big)^k \! \Big[\frac{1}{m}\!\sum_{y\in S^1(\emptyset)}S_h(y)+c_{\beta}h(\emptyset)\Big]\!+\! \frac{1}{c_{\beta}}
			\!\sum_{j=1}^{k}
			\Big(\frac{\beta}{1-\beta}\Big)^{k-j}\! S_h(x^j).
		\end{array}	
	\end{equation}
	
	Since $\beta < \frac12$, it holds that 
	$$
	\lim_{k \to \infty} \Big(\frac{\beta}{1-\beta}\Big)^k\int_0^1 f(s) ds =0
	$$
	and also 
	$$
	\lim_{k \to \infty} \Big(\frac{\beta}{1-\beta}\Big)^k \! \Big[\frac{1}{m}\!\sum_{y\in S^1(\emptyset)}S_h(y)+c_{\beta}h(\emptyset)\Big]\! = 0.
	$$
	
	Now, given $\gamma>0$, since $f$ is continuous, we have that 
	there exists $k_0$ such that
	$$
	\left| \itx{x^j} - f(\psi(z))\right| < \gamma
	$$
	for every $j>k_0$. Now, we use that $$\cb=\sum_{l=0}^{\infty}\Big(\frac{\bt}{1-\bt}\Big)^{l}$$
	and that $f$ is bounded to obtain 
	\begin{equation}
		\label{Formuladeu.88}
		\begin{array}{ll}
		\displaystyle 
			\left| \frac{1}{c_{\beta}}\sum_{j=1}^{k}\Big(\frac{\beta}{1-\beta}\Big)^{k-j}\itx{x^j}- f(\psi(z)) \right| 
			\\[10pt]
	\displaystyle		\qquad = 
	\left| \frac{1}{c_{\beta}}\sum_{j=1}^{k}\Big(\frac{\beta}{1-\beta}\Big)^{k-j}\itx{x^j}- 
	\frac{1}{c_{\beta}}\sum_{l=0}^{\infty}\Big(\frac{\beta}{1-\beta}\Big)^{l}f(\psi(z)) \right|
				\\[10pt]
	\displaystyle		\qquad \leq 
	C \sum_{j=1}^{k_0-1}\Big(\frac{\beta}{1-\beta}\Big)^{k-j} \\[10pt]
	\displaystyle		\qquad \qquad +
	\left| \frac{1}{c_{\beta}}\sum_{l=0}^{k-k_0}\Big(\frac{\beta}{1-\beta}\Big)^{l}\itx{x^{k-l}}- 
	\frac{1}{c_{\beta}}\sum_{l=0}^{\infty}\Big(\frac{\beta}{1-\beta}\Big)^{l}f(\psi(z)) \right|
\\[10pt]
	\displaystyle		\qquad\leq 
	C \sum_{j=1}^{k_0-1}\Big(\frac{\beta}{1-\beta}\Big)^{k-j} + 
	\gamma \frac{1}{c_{\beta}}\sum_{l=0}^{k-k_0}\Big(\frac{\beta}{1-\beta}\Big)^{l} 
	+  C \frac{1}{c_{\beta}}\sum_{l=k-k_0+1}^{\infty}\Big(\frac{\beta}{1-\beta}\Big)^{l} 
	\leq \frac{\varepsilon}{2},
		\end{array}	
	\end{equation}
	for $\gamma$ small enough and $k$ large enough. 
Hence, we have that
$$
\lim_{k\to \infty} \left| \frac{1}{c_{\beta}}\sum_{j=1}^{k}\Big(\frac{\beta}{1-\beta}\Big)^{k-j}\itx{x^j}- f(\psi(z)) \right|  = 0.
$$
	
	Finally, the condition
	$$
	\lim_{x\in \T, |x|=k \to \infty} \sum_{j=1}^{k}
			\Big(\frac{\beta}{1-\beta}\Big)^{k-j}\! S_h(x^j) = 0,
	$$
	is just what we need to ensure that 
	$$
	\lim_{x\rightarrow z\in\partial\T}u(x)=f(\psi(z)).
	$$
	This shows that there exists a solution. Uniqueness follows for our previous arguments. Note that if $u$ is a solution, $u$ is given by \eqref{Formuladeu}. It fact \eqref{Formuladeu} is a representation formula for the solution in terms of $h$ and $f$. 
	 Uniqueness also 
	follows form the comparison principle
	(see Lemma \ref{lema-compar} below). 
	\end{proof}
	
	\begin{remark}
		Under our hypothesis on $\beta$ and $h$, for $f:[0,1]\rightarrow\R$ continuous the function $w_1:\T\rightarrow\R$ defined as
		\begin{equation}
			\begin{array}{ll}
			\displaystyle w_1(x)=\Big(\frac{\beta}{1-\beta}\Big)^k\Big[\frac{1}{m}\sum_{y\in S^1(\emptyset)}S_h(y)+c_{\beta}h(\emptyset)\Big] 
			 \\[10pt]
			\displaystyle \qquad \qquad \qquad+\frac{1}{c_{\beta}}\sum_{j=1}^{k}
			\Big(\frac{\beta}{1-\beta}\Big)^{k-j}S_h(x^j), \quad x\neq\emptyset, |x|=k, 
			\\[10pt]
			\displaystyle w_1(\emptyset)=\frac{1}{m}\s{y}{\emptyset}{1}\Sh{y}+\cb h(\emptyset),
			\end{array}	
		\end{equation}
	is the solution to
	\begin{equation}
		\left\{
		\begin{array}{l}
			\displaystyle 
			L(w_1)(x) = h(x), \qquad \;x\in \T,\\[10pt]
			\displaystyle   	\lim_{x\rightarrow z\in\partial\T}w_1(x)= 0. 
		\end{array} \right.
	\end{equation}
	Moreover, the function $w_2:\T\rightarrow\R$ defined as 
	\begin{equation}
		\begin{array}{ll}
			\displaystyle 	w_2(x)=\Big(\frac{\beta}{1-\beta}\Big)^k\int_0^1 f(s) ds+\frac{1}{c_{\beta}}\sum_{j=1}^{k}\Big(\frac{\beta}{1-\beta}\Big)^{k-j}\itx{x^j}, \quad x\neq\emptyset, |x|=k,
			\\[10pt]
			\displaystyle w_2(\emptyset)=\int_{0}^{1}f(s){\rm d}s,
		\end{array}
	\end{equation}
	is the solution to
\begin{equation}
	\left\{
	\begin{array}{l}
		\displaystyle 
		L(w_2)(x)=0, \;\;x\in \T,\\[10pt]
		\displaystyle
		\lim_{x\rightarrow z\in\partial\T}w_2(x)= f(\psi(z)). 
	\end{array} \right.
\end{equation}

Therefore, we can write 
$$
u(x) = w_1(x) + w_2(x)
$$
being $w_1$ the solution to \eqref{EQL} for a given $h$ with $f=0$ and $w_2$ the solution for a given $f$ with $h=0$.
	\end{remark}

	\begin{remark} \label{rem.33}
	Notice that analogous arguments allows us to show that, under the same conditions, we can build supersolutions as large as we want
	enlarging $h$ (and also large subsolutions) 
	to the equation that satisfy the boundary condition
	$$
	\lim_{x\rightarrow z\in\partial\T}u(x)= f(\psi(z)).
    $$
To this end, consider $L$ defined as in \eqref{L1} and \eqref{L2} and $h_{1}, h_{2}:\T \rightarrow \mathbb{R}$ two functions in the tree such that $h_{1}\leq h_{2}$ and $f$ a boundary continuous function.  Moreover, suppose that  $\beta$, $h_{1}$ and $h_{2}$ satisfies the solvability condition. Let $u_{i}$ be 
the unique solution for the Dirichlet problem of the single equation asociated to the operator $L-h_i$ and boundary function $f$, that is, $u_{i}$ is given by \eqref{solSE} for $i=1,2$. Then, $u_{1}$ a supersolution the Dirichlet problem with $L-h_2$ and  $u_{2}$ is a subsolution for  the Dirichlet problem with $L-h_1$ and $f$. Now, notice that we can make $u_1$
as large as we want in a finite number of nodes just taking $h_1 (\emptyset)$ large enough. 
\end{remark}

	Let us prove the uniqueness for solutions to \eqref{EQL}. To this end we prove that the operator $L$ verifies the comparison principle.
	
	\begin{lemma}{(Comparison Principle)}
	\label{lema-compar}
		Given $u,v:\T\rightarrow\R$ such that
		\begin{equation}
			\left\lbrace
			\begin{array}{ll}
				\displaystyle L(u)(x)\geq h(x), \quad  & x\in \T, \\[10pt]
				\displaystyle \liminf_{x\rightarrow z\in\partial\T}u(x) \geq f(\psi(z)).
			\end{array}
			\right.
		\end{equation}
	and 
	\begin{equation}
		\left\lbrace
		\begin{array}{ll}
			\displaystyle L(v)(x)\leq h(x), \quad  & x\in \T, \\[10pt]
			\displaystyle \limsup_{x\rightarrow z\in\partial\T}v(x) \leq g(\psi(z)).
		\end{array}
		\right.
	\end{equation}
Then, if $f\geq g$ in $[0,1]$, we get $$u(x)\geq v(x), \qquad x \in\T.$$
	\end{lemma}
	\begin{proof}
		Let us argue by contradiction, i.e., assume that
		$$
		\sup_{x\in\T}(v-u)(x)=\theta>0.
		$$
		 Let us suppose that there exists $k\geq 1$, and $|x_0|=k$ such that $(v-u)(x_0)\geq \frac{\theta}{2}$ and $(v-u)(y)<\frac{\theta}{2}$ for all $|y|<k$. At that point we have
		\[
			L(v) (x_0) - L(u)(x_0)\leq 0.
		\]
		Solving this inequality we arrive to
		\[
		\begin{array}{ll}
		\displaystyle 0<\frac{\theta}{2}\leq (v-u)(x_0)\leq \beta(v-u)(\hat{x_0})+\frac{1-\beta}{m}\sum_{y\in S(x_0)}(v-u)(y) \\
		\displaystyle \qquad \leq  \beta(v-u)(x_0)+\frac{1-\beta}{m}\sum_{y\in S(x_0)}(v-u)(y),
		\end{array}
		\]
		then
		\[
		\displaystyle (v-u)(x_0)\leq \frac{1}{m}\sum_{y\in S(x_0)}(v-u)(y).
		\]
		This implies that there exists $x_1\in S(x_0)$ such that $(v-u)(x_1)\geq (v-u)(x_0)$. If we repeat this argument we will obtain a sequence $(x_n)_{n\geq 0}$ such that $x_{n+1}\in S(x_n)$ and $(v-u)(x_{n+1})\geq (v-u)(x_n)$. Thus, we get
		\[
		\begin{array}{l}
		\displaystyle 
		0<\frac{\theta}{2}\leq \liminf_{x_k\rightarrow z\in\partial\T}(v-u)(x_k)\\[10pt]
		\displaystyle \qquad \leq \limsup_{x_k\rightarrow z\in\partial\T}v(x_k)-\liminf_{x_k\rightarrow z\in\partial\T}u(x_k) \leq g(\psi(z))-f(\psi(z))\leq 0,
		\end{array}
		\]
		which is a contradiction. Now, let us consider the case $$(v-u)(\emptyset)\geq \frac{\theta}{2}.$$ In this case we have
		\[
		0<\frac{\theta}{2}\leq (v-u)(\emptyset)\leq \frac{1}{m}\sum_{y\in S(\emptyset)}(v-u)(y),
		\]
		and this implies that there exists $x_1\in S(\emptyset)$ such that $(v-u)(x_1)\geq(v-u)(\emptyset)$. If from this point we argue as before we obtain a contradiction. This ends the proof.
	\end{proof}

	%%%%%%%%%%%%%%%%%%%%%%%%%%%%%%%%%%%%%%%%%%%%%%%%%%%%%%%%%%%%%%%%%%%%%%%%%%%%%%%%%%%%
	%%%%%%%%%%%%%%%%% OBSTECLE PROBLEM %%%%%%%%%%%%%%%%%%%%%%%%%%%%%%%%%%%%%%%%%%%%%%%%%
	%%%%%%%%%%%%%%%%%%%%%%%%%%%%%%%%%%%%%%%%%%%%%%%%%%%%%%%%%%%%%%%%%%%%%%%%%%%%%%%%%%%%
	
	\section{The two membranes problem.} 
	\label{sect-membranes}
	
	\subsection{The obstacle problem on trees.}

	We defined a solution for the obstacle problem from below and from above
	in the introduction, see Definition \ref{def.obs.intro}.
	Now, we claim that
	the problem can be regarded from two perspectives.

	\begin{definition}
			Given an operator $L$ defined as before and a function $h : \T  \rightarrow \mathbb{R}$,
			such that the conditions for
			solvability, \eqref{cond.una.ecuacion.intro}, hold, $\varphi:\T\rightarrow\R$ a bounded function, and $f:[0,1]\rightarrow\R$ a continuous function such that $$\limsup_{x\rightarrow z\in\partial\T}\varphi(x)<f(\psi(z)),$$ we say that $u$ is a solution to the obstacle problem from below, and we note $u=\ol(L,h,\varphi,f)$ if $u$ verifies
			\begin{equation} \label{obst.u.sect}
	\displaystyle u(x) = \inf \left\{ w(x) : 
	\begin{array}{l} L(w)(x)\geq h(x), \quad x\in \T, \\[6pt]
\displaystyle w(x) \geq \varphi (x) , \quad x\in \T, \\[6pt]
	 \displaystyle \mbox{ and } \liminf_{x\rightarrow z\in\partial\T}w(x)\geq f(\psi(z))
	 \end{array} \right\},	
			\end{equation}
that is equivalent to
		\begin{equation}
			\label{DefObstbelow}
		\left\lbrace
		\begin{array}{ll}
		\displaystyle u(x)\geq \varphi(x), \quad  & x\in \T, \\[10pt]
		\displaystyle L(u)(x)\geq h(x), \quad & x\in \T, \\[10pt]
		\displaystyle L(u)(x) = h(x), \quad  & x\in \{u>\varphi\},\\[10pt]
		\displaystyle \lim_{x\rightarrow z\in\partial\T}u(x)= f(\psi(z)).
		\end{array}
		\right.
		\end{equation}
		Notice that \eqref{DefObstbelow} can be written as
		$$
		\left\{
		\begin{array}{l}
		\displaystyle
			0=
\max\Big\{  -L (u)({x})+h(x),
\varphi (x)-u(x)  \Big\},\quad x\in \T,
\\[10pt]
		\displaystyle \lim_{x\rightarrow z\in\partial\T}u(x)= f(\psi(z)).
\end{array}
\right. $$
		
		Analogously, for $g:[0,1]\rightarrow\R$ and $\varphi:\T\rightarrow\R$ such that $$\liminf_{x\rightarrow z\in\partial\T}\varphi(x)>g(\psi(z))$$ we define $v$ a solution from above and denote 
		$v=\ou(L,h,\varphi,g)$ if
		\begin{equation} \label{obst.v.sect}
	\displaystyle v(x) = \sup \left\{ w(x) : 
	\begin{array}{l} L(w)(x)\leq h(x), \quad x\in \T, \\[6pt]
	 \displaystyle w(x) \leq \varphi(x) , \quad x\in \T,
	 \\[6pt]
	 \displaystyle \mbox{ and } \limsup_{x\rightarrow z\in\partial\T}w(x)\leq g(\psi(z)) 
	 \end{array} \right\},
			\end{equation}
that is equivalent to
		\begin{equation}
			\label{DefObstabove}
		\left\lbrace
		\begin{array}{ll}
		\displaystyle v(x)\leq \varphi(x), \quad  & x\in \T, \\[10pt]
		\displaystyle L(v)(x)\leq h(x), \quad & x\in \T, \\[10pt]
		\displaystyle L(v)(x) = h(x), \quad  & x\in \{v<\varphi\},\\[10pt]
		\displaystyle \lim_{x\rightarrow z\in\partial\T}v(x)= g(\psi(z)),
		\end{array}
		\right.
		\end{equation}
	that can be written as
	$$
		\left\{
		\begin{array}{l}
		\displaystyle
			0=
\min\Big\{  -L( v)({x})+h(x),
\varphi (x)-v(x) \Big\},\quad x\in\T,
\\[10pt]
		\displaystyle \lim_{x\rightarrow z\in\partial\T}v(x)= g(\psi(z)).
\end{array}
\right. $$
\end{definition}
	
	Let us prove that both definitions (the one as inf/sup of super/subsolutions and the one solving inequalities) of being a solution to the obstacle problem
	are equivalent. 
	\begin{proposition}
		Given an operator $L$ defined as in \eqref{L1} and \eqref{L2}, and a function $h: \T \to \mathbb{R}$
		such that the conditions for
		solvability, \eqref{cond.una.ecuacion.intro}, hold, $\varphi:\T\rightarrow\R$ a bounded function, and $f:[0,1]\rightarrow\R$ a continuous function such that 
		$$
		\limsup_{x\rightarrow z\in\partial\T}\varphi(x)<f(\psi(z)),
		$$
		 the function defined by \eqref{obst.u.sect} is well defined and is solution to \eqref{DefObstbelow}. Conversely, a solution to \eqref{DefObstbelow}
		 is the minimizer in \eqref{obst.u.sect}.
		 We denote by $u=\ol(L,h,\varphi,f)$ the solution to 
		 \eqref{obst.u.sect} or to \eqref{DefObstbelow}.
		 
		 Analogously, given $g:[0,1]\rightarrow\R$ a continuous function such that 
		 $$
		 \liminf_{x\rightarrow z\in\partial\T}\varphi(x)>g(\psi(z)),
		 $$
		 the function defined by \eqref{obst.v.sect} is well defined and is solution to \eqref{DefObstabove}. 
		 Conversely, a solution to \eqref{DefObstabove}
		 is the maximizer in \eqref{obst.v.sect}.
		 We denote the solution to \eqref{DefObstabove}
		 or the maximizer in \eqref{obst.v.sect} by $v=\ou(L,h,\varphi,g)$.
	\end{proposition}

\begin{proof}
	Let us consider the set
	\begin{equation}
		\label{SETbelow}
		\underline{\Lambda}_{f,h,\varphi}=\Big\{w : L(w)\geq h \ , \ w\geq \varphi \ , \ \liminf_{x\rightarrow z\in\partial\T}w(x)\geq f(\psi(z))   \Big\}.	
	\end{equation}
	
	This set is non empty and is bounded from below. In fact, if we consider $M=\max\{\lVert f\rVert_{\infty},\lVert \varphi\rVert_{\infty} \}$, and $w_0$ the unique bounded solution to
	\begin{equation}
		\left\lbrace\begin{array}{ll}
			\displaystyle L(w_0)(x)=h(x), & x\in\T,\\[10pt]
			\displaystyle \lim_{x\rightarrow z\in\partial\T}w_0 =f(\psi(z)).
		\end{array}
		\right.
	\end{equation} 
	Note that $\underline{\Lambda}_{f,h,\varphi}$ is not empty due to the fact that the function $w_0+M\in\underline{\Lambda}_{f,h,\varphi}$. Moreover,  if $w\in\underline{\Lambda}_{f,h,\varphi}$, $w(x)\geq\varphi(x)\geq -M$, then, is bounded from below. Let us define
	\begin{equation}
		\label{caracu}
		u(x)=\inf_{w\in\underline{\Lambda}_{f,h,\varphi}}w(x).
	\end{equation}
	
	Let us show that this function $u$ verifies \eqref{DefObstbelow}. In fact, using that $w\geq \varphi$ for all $w\in\underline{\Lambda}_{f,h,\varphi}$, taking infimum we get $u\geq \varphi$. In addition, given $w\in\underline{\Lambda}_{f,h,\varphi}$ we have 
	\[
	L(w)(x)\geq h(x) \quad \Rightarrow \quad w(x)\geq \beta w(\hat{x})+(1-\beta)\Big(\frac{1}{m}\sum_{y\in S^1 (x)} w(y)\Big)+h(x), \quad x\neq\emptyset.
	\]
	As a consequence, if we take infimum in the right hand of the above inequalty we obtain
	\[
	w(x)\geq \beta u(\hat{x})+(1-\beta)\Big(\frac{1}{m}\sum_{y\in S^1 (x)} u(y)\Big)+h(x), \quad x\neq\emptyset.
	\]
	Furthermore,  taking infimum in the left hand 
	\[
	u(x)\geq \beta u(\hat{x})+(1-\beta)\Big(\frac{1}{m}\sum_{y\in S^1 (x)} u(y)\Big)+h(x) \quad \Rightarrow \quad L(u)(x)\geq h(x), \quad x\neq\emptyset.
	\]
	Analogously, we can do the same computation on the root $\emptyset$ and obtain
	\[
	u(\emptyset)\geq \Big(\frac{1}{m}\sum_{y\in S^1 (\emptyset)} u(y)\Big)+h(\emptyset) \quad \Rightarrow \quad L(u)(\emptyset)\geq h(x).
	\]
	Finally, we have $$\liminf_{x\rightarrow z\in\partial\T}w(x)\geq f(\psi(z))$$ for all $w\in\underline{\Lambda}_{f,h,\varphi}$, taking infimum we get $$\liminf_{x\rightarrow z\in\partial\T}u(x)\geq f(\psi(z)).$$ We just proved that $u\in\underline{\Lambda}_{f,h,\varphi}$. Let us prove now that $$L(u)(x)=0$$ if $x\in\{u>\varphi\}$. Suppose that this is not true, given $x_0\in\{u>\varphi\}$ such that $L(u)(x_0)> h(x_0)$, i.e. 
	$$u(x_0)>\beta u(\hat{x})+(1-\beta)\Big(\frac{1}{m}\sum_{y\in S^1 (x)} u(y)\Big)+h(x) .$$ Let us define $$\delta_1=u(x_0)-\beta u(\hat{x})+(1-\beta)\Big(\frac{1}{m}\sum_{y\in S^k (x)} u(y)\Big)+h(x), $$ $$\delta_2=u(x_0)-\varphi(x_0)$$ and $$\delta_0=\min\{\delta_1,\delta_2\}.$$ If we consider
	\[
	u_0(x)=\left\lbrace\begin{array}{ll}
		\displaystyle u(x), \quad & x\neq x_0,\\[8pt]
		\displaystyle u(x)-\frac{\delta_0}{2}, \quad & x=x_0. 
	\end{array}
	\right.
	\]     
	This function verifies $u_0\geq \varphi$, $$\liminf_{x\rightarrow z\in\partial\T}u_0(x)\geq f(\psi(z))$$ and $L(u_0)(x)\geq h(x)$. In fact, $L(u_0)(x_0)\geq h(x)$ and 
	then $L(u_0)(\hat{x_0})\geq h(\hat{x_0})$, and for $y_0\in S(x_0)$, $L(u_0)(y_0)\geq h(y_0)$. Thus, $u_0\in \underline{\Lambda}_{f,h,\varphi}$ and $u_0<u$ which is a contradiction. We have proved that $$L(u)(x)=h(x)$$ in the set $\{u>\varphi\}$. 
	
	Let us verify that $$\limsup_{x\rightarrow z\in\partial\T}u(x)\leq f(\psi(z)).$$ Suppose that this is not true, i.e. 
	$$\limsup_{x\rightarrow z\in\partial\T}u(x)> f(\psi(z)).$$ Using that $$\limsup_{x\rightarrow z\in\partial\T}\varphi(x) < f(\psi(z)),$$ there exists $k\in\mathbb{N}$ such that $|x|\geq k$ and $\varphi(x) < f(\psi(z))$. Then, there exists $x_0$ such that $|x_0|\geq k$ and $u(x_0)>f(\psi(z))$. We can argue as before to make a function $u_0\in\underline{\Lambda}_{f,h,\varphi}$ and $u_0<u$ which is a contradiction. 
	
	Conversely, assume that $u$ solves \eqref{DefObstbelow} and let us prove that
	$u$ is the minimizer in \eqref{obst.u.sect}. 
	Since $u$ solves \eqref{DefObstbelow} we have that it satisfies
	$L(u)(x)\geq h(x)$ for $x\in \T$, $
	\lim_{x\rightarrow z\in\partial\T}u(x) = f(\psi(z))$ and $w(x) \geq \varphi(x)$
	and then $u$ is a competitor in the minimization problem \eqref{obst.u.sect}.
	Therefore,
	$$
	u(x) \geq \inf_{w\in\underline{\Lambda}_{f,h,\varphi}}w (x).
	$$
	
	To prove the reverse inequality, call $z$ the minimizer, that is, 
	$$
	z(x) = \inf_{w\in\underline{\Lambda}_{f,h,\varphi}}w (x).
	$$ 
	Since we have that $u\geq z$ we get an inclusion for the sets where these
	functions touch the obstacle,
	$$
	\{x : u(x) =\varphi (x) \} \subset \{x : z(x) =\varphi (x) \}.
	$$
	Outside the set $\{x : u(x) =\varphi (x) \} $ we have that $u$ solves 
	$L(u)(x)= h(x)$ and $z$ is a supersolution to this equation, while the boundary conditions
	give
	$$
	\liminf_{x\rightarrow z\in\partial\T}z(x) \geq  f(\psi(z)) = 
	\lim_{x\rightarrow z\in\partial\T}u(x) .
	$$
	From the proof of the comparison principle, we get that
	$$
	u(x) \leq z(x),
	$$
	and we conclude that $u\equiv z$, the minimizer of \eqref{obst.u.sect}. 
	
	Analogously, we consider
	\begin{equation}
		\label{SETabove}
		\overline{\Lambda}_{g,h,\varphi}=\Big\{w : L(w)\leq h \ , \ w\leq \varphi \ , \ \limsup_{x\rightarrow z\in\partial\T}w(x)\leq g(\psi(z))  \Big\}.	
	\end{equation}
	
	This set is non empty and bounded from above. Then, we define
	\begin{equation}
		\label{caracv}
		v(x)=\sup_{w\in\overline{\Lambda}_{g,h,\varphi}}w(x),
	\end{equation}
	
	This function $v$ verifies \eqref{DefObstabove}, i.e $v=\ou(L,h,\varphi,g)$, 
	and conversely, a function that verifies  \eqref{DefObstabove} is given by
	\eqref{caracv}. This ends the proof.
\end{proof}

\begin{corollary}
	Given an operator $L$ defined as in \eqref{L1} and \eqref{L2},
	and a function $h:\T \to \mathbb{R}$
	such that the conditions for
	solvability, \eqref{cond.una.ecuacion.intro}, hold, $\varphi:\T\rightarrow\R$ a bounded function, and $f:[0,1]\rightarrow\R$ a continuous function. If  $$\limsup_{x\rightarrow z\in\partial\T}\varphi(x)<f(\psi(z)),$$ 
	there exists a unique solution for the obstacle problem from below, 	
	\begin{equation}
		\ol(L,h,\varphi,f)=\inf_{w\in\underline{\Lambda}_{f,h,\varphi}}w.
	\end{equation}
	where $\underline{\Lambda}_{f,h,\varphi}$ is defined in \eqref{SETbelow}.
	
	Analogously, if $$\liminf_{x\rightarrow z\in\partial\T}\varphi(x)>f(\psi(z)),$$ 
	there exists a unique solution for the obstacle problem form above,	
	\begin{equation}
		\ou(L,h,\varphi,g)=\sup_{w\in\overline{\Lambda}_{f,h,\varphi}}w.
	\end{equation}
	where $\overline{\Lambda}_{g,h,\varphi}$ is defined in \eqref{SETabove}.
\end{corollary}

\begin{proof}
	 The existence of $\ol(L,h,\varphi,f)$ and of $\ou(L,h,\varphi,g)$ relies on the fact that the sets $\underline{\Lambda}_{f,h,\varphi}$  and $\overline{\Lambda}_{g,h,\varphi}$ are not empty, and there exists a $M>0$ sucht that $w\geq -M$ for all $w\in \underline{\Lambda}_{f,h,\varphi}$ and $v\leq M$ for all $v\in \overline{\Lambda}_{f,h,\varphi}$, as we showed in the proof of the equivalence of the definitions of being a solution to the obstacle problem. The uniqueness is a direct consequence of taking the infimum or the supremum on these sets.
\end{proof}

%%%%%%%%%%%%%%%%%%%%%%%%%%%%%%%%%%%%%%%%%%%%%%%%%%%%%%%%%%%%%%%%%%%%%%%%%%
%%%%%%%%%%%%%%%% ITERATIVE METHOD %%%%%%%%%%%%%%%%%%%%%%%%%%%%%%%%%%%%%%%%
%%%%%%%%%%%%%%%%%%%%%%%%%%%%%%%%%%%%%%%%%%%%%%%%%%%%%%%%%%%%%%%%%%%%%%%%%%

\subsection{The two membranes problem. Existence via iterations of the obstacle problem.}
\label{subsect-iteraciones} \ \\

Let us show that the two membranes problem in the tree has a solution. 

\begin{proof}[Proof of Theorem \ref{teo.main.intro}]
 Let us consider $L_{1}$ and $L_{2}$ two operators given by de mean value formula
 and $h_1,h_2: \T \to \mathbb{R}$ that verify the
 solvability condition with $\beta_1$, $h_1$ and $\beta_2$, $h_2$ respectively, and $f,g:[0,1]\rightarrow\R$ two continuous functions such that $f> g$. Let us start the method with $v_0$ a bounded subsolution of the operator with boundary condition $g$, that is,
\begin{equation}
	\begin{array}{ll}
		\displaystyle L_2(v_0)(x)\leq h_2 (x), & x\in\T, \\[10pt]
		\displaystyle \limsup_{x\rightarrow z\in\partial\T}v_0(x)\leq g(\psi(z)).
	\end{array}
\end{equation} 	
With this function $v_0$ we let $u_1$ be the solution to the obstacle problem from below
for the operator $L_{1}$, right hand side $h_1$, boundary datum $f$ and obstacle $v_0$, that is,
$$
u_1=\ol(L_{1},h_1,v_0,f).
$$
Notice that the set
$$
\underline{\Lambda}_{f,h_1v_0}=\Big\{w : L_1(w)\geq h_1 \ , \ w\geq v_0 \ , \ \liminf_{x\rightarrow z\in\partial\T}w(x)\geq f(\psi(z))   \Big\}
$$
is not empty, since by the results in Section \ref{sect-soluciones} we can construct a supersolution
to our problem with $L_1-h_1$ and boundary condition $f$ as large as we want (see Remark
\ref{rem.33}). Therefore, the function $u_1$ can be obtained solving the minimization 
problem $\inf\{w(x):\;w\in\underline{\Lambda}_{f,h_1v_0}\}.$

With this function $u_1$, we take
$$
v_1=\ou(L_{2},h_2,u_{1},g).
$$
Here, notice that the corresponding set 
$$
\overline{\Lambda}_{g,h_2,u_1}=\Big\{w : L_2(w)\leq h_2 \ , \ w\leq u_1 \ , \ \limsup_{x\rightarrow z\in\partial\T}w(x)\leq g(\psi(z))   \Big\}
$$
is not empty (here we recall that we can construct large subsolutions to $L_2-h_2$) and hence
$v_1$ can be obtained from the maximization problem 
$\sup\{w(x):\; w\in\overline{\Lambda}_{g,h_2,u_1}\}$

Now, we iterate this procedure and define
\begin{equation}
	u_n=\ol(L_{1},h_1,v_{n-1},f) \quad \mbox{ and } \quad v_n=\ou(L_{2},h_2,u_{n},g).
\end{equation}
In this way we obtain two sequences $\{u_n\}_{n\geq 1},\{v_n\}_{n\geq 1}$.
Our goal is to show that these sequences are monotone and that they converge to
a pair of functions $(u,v)$ that is a solution to the two membranes problem. 

\medskip

CLAIM \# 1: The sequences are increasing, i.e. $u_n\geq u_{n-1},v_n\geq v_{n-1}$. 

Let us start with $v_{n-1}$. By definition of being a solution to the obstacle problem,  $v_{n-1}$ satisfies $u_{n-1}\geq v_{n-1}$ , $L_{2}(v_{n-1})\leq h_2$ and $\lim_{x\rightarrow z\in\partial\T}v_{n-1}(x)= g(\psi(z))$. Hence, $v_{n-1} \in \overline{\Lambda}_{g,h_2,u_{n-1}}$. Then, using again the definition of being a solution to the obstacle problem we get $$v_n=\sup_{w\in\overline{\Lambda}_{g,h_2,u_{n-1}}}w\geq v_{n-1}.$$ 

On the other hand, $u_n\geq v_n\geq v_{n-1}$, $L_{1}(u_n)\geq h_1$ and $\lim_{x\rightarrow z\in\partial\T}u_{n}(x)= f(\psi(z))$, then, using one more time the definition of being a solution to the obstacle problem, we get $$u_{n-1}=\inf_{w\in\underline{\Lambda}_{f,h_1,v_{n-1}}}w\leq u_n.$$ 
This ends the proof of the claim.

\medskip

CLAIM \# 2: The sequences are bounded. 

Let us start with $\{u_n\}_{n\geq 1}$. Let $w$ be the solution to 
\begin{equation}
		\left\lbrace
		\begin{array}{ll}
			\displaystyle L_{2}(w)(x)= h_2 (x), \quad  & x\in \T,\\[10pt]
			\displaystyle \lim_{x\rightarrow z\in\partial\T}w(x) = g(\psi(z)).
		\end{array}
		\right.
	\end{equation}
	By the comparison principle $v_n\leq w$. Let us consider $$\rho=\ol(L_{1},h_1,w,f).$$
	Note that $\rho$ is bounded since $w$ is bounded and the operators
	satisfy the solvability conditions.
	 Since $v_n\leq w$ we have that $\rho\in\underline{\Lambda}_{g,h_2,v_n}$ for all $n\in\mathbb{N}$. Using \eqref{caracu} we get $$u_n\leq \rho$$ for all $n\in\mathbb{N}$.

	 Now we observe that $\{v_n\}_{n\geq 1}$ is bounded, since it holds that $$v_n\leq u_{n-1}\leq \rho$$ for all $n\in\mathbb{N}$. This ends the proof the second claim.

\medskip	 
	 
Thanks to CLAIM \# 1 and CLAIM \# 2 we can take limits as $n\to \infty$ to obtain
\begin{equation}
	u_n(x) \rightarrow u_{\infty}(x) \quad \mbox{ and } \quad v_n(x) \rightarrow v_{\infty}(x) ,
\end{equation}
for all $x\in \T$.

\medskip

CLAIM \# 3: the limit pair $(u_\infty, v_\infty)$ solves the two membranes problem.

Let us prove that	
\begin{equation}
	u_{\infty}=\ol(L_1,h_1,v_{\infty},f) \quad \mbox{ and } \quad v_{\infty}=\ou(L_2,h_2,u_{\infty},g).
\end{equation}	

First, we observe that we have 
$$
\left\{
\begin{array}{l}
\displaystyle
u_n(x) \! = \!
\max \!
\Big\{ \!\beta_1 u_n(\hat{x})+(1-\beta_1)\Big(\frac{1}{m}\!\! \sum_{y\in S^1 (x)} u_n(y)\Big)\!+ \! h_1(x),
v_{n-1} (x) \Big\} ,\ x \neq \emptyset,
\\[10pt]
\displaystyle u_n(\emptyset)\! = \!
\max \!
\Big\{ \! \Big(\frac{1}{m}\sum_{y\in S^1 (\emptyset)} u_n(y)\Big)+h_1(\emptyset),
v_{n-1} (x) \Big\} ,
\\[10pt]
\displaystyle  \lim_{x\rightarrow z\in\partial\T}u_n(x)= f(\psi(z)).
\end{array}
\right. 
$$

Therefore, from the monotonicity of the sequence $u_n$ we get
$$
 u_\infty (x) \geq u_n (x)\geq
 \beta_1 u_n(\hat{x})+(1-\beta_1)\Big(\frac{1}{m}\sum_{y\in S^1 (x)} u_n(y)\Big)+h_1(x),
$$
$$
 u_\infty (\emptyset) \geq u_n (\emptyset)\geq
\Big(\frac{1}{m}\sum_{y\in S^1 (\emptyset)} u_n(y)\Big)+h_1(\emptyset),
$$
and
$$
u_\infty (x) \geq u_n (x)\geq v_{n-1}(x). 
$$
Letting $n \to \infty$ in the right hand sides we get 
\begin{equation} \label{desig.infty.u}
\begin{array}{l}
\displaystyle 
 u_\infty (x) \geq \beta_1 u_\infty(\hat{x})+(1-\beta_1)\Big(\frac{1}{m}\sum_{y\in S^1 (x)} u_\infty (y)\Big)+h_1(x), \\[10pt]
 \displaystyle
 u_\infty (\emptyset) \geq \Big(\frac{1}{m}\sum_{y\in S^1 (\emptyset)} u_\infty (y)\Big)+h_1(\emptyset),
 \end{array}
\end{equation}
and
\begin{equation} \label{desig.infty.u.v}
u_\infty (x) \geq v_\infty(x). 
\end{equation}

Now, for $v_n$ we have 
$$
v_n (x) \leq 
 \beta_2 v_n (\hat{x})+(1-\beta_2)\Big(\frac{1}{m}\sum_{y\in S^1 (x)} v_n (y)\Big)+h_{2}(x),
$$
$$
v_n (\emptyset) \leq 
\Big(\frac{1}{m}\sum_{y\in S^1 (\emptyset)} v_n (y)\Big)+h_{2}(\emptyset),
$$
and using again the monotonicity of the sequence we get (after passing to the limit
as $n \to \infty$),
\begin{equation} \label{desig.infty.v}
\begin{array}{l}
\displaystyle
v_\infty (x) \leq 
 \beta_2 v_\infty (\hat{x})+(1-\beta_2)\Big(\frac{1}{m}\sum_{y\in S^1 (x)} v_\infty (y)\Big)+h_{2}(x),\quad x \neq \emptyset, \\[10pt]
 \displaystyle 
 v_\infty (\emptyset) \leq 
 \Big(\frac{1}{m}\sum_{y\in S^1 (\emptyset)} v_\infty (y)\Big)+h_{2}(\emptyset).
 \end{array}
\end{equation}

Now, for a point $x$ in the set $\{ x: u_\infty (x) > v_\infty (x)\}$ with $x\neq \emptyset$
we have that
$u_n (x) > v_n (x)$ for $n$ large and hence we obtain, for 
$x\in \{ x: u_\infty (x) > v_\infty (x)\}$,
$$
u_n (x)=
 \beta_1 u_n(\hat{x})+(1-\beta_1)\Big(\frac{1}{m}\sum_{y\in S^k (x)} u_n(y)\Big)+h_1(x),
$$
and 
$$
v_n (x) = 
 \beta_2 v_n (\hat{x})+(1-\beta_2)\Big(\frac{1}{m}\sum_{y\in S^k (x)} v_n (y)\Big)+h_{2}(x).
$$
Passing to the limit as $n \to \infty$ we conclude that 
\begin{equation} \label{ig.infty.u}
 u_\infty (x) = \beta_1 u_\infty(\hat{x})+(1-\beta_1)\Big(\frac{1}{m}\sum_{y\in S^1 (x)} u_\infty (y)\Big)+h_1(x),
\end{equation}
and 
\begin{equation} \label{ig.infty.v}
v_\infty (x) = 
 \beta_2 v_\infty (\hat{x})+(1-\beta_2)\Big(\frac{1}{m}\sum_{y\in S^1 (x)} v_\infty (y)\Big)+h_{2}(x).
\end{equation}

An analogous computation can be done when $u_\infty (\emptyset) > v_\infty (\emptyset)$
to obtain  
\begin{equation} \label{ig.infty.u.22}
 u_\infty (\emptyset) =\Big(\frac{1}{m}\sum_{y\in S^1 (\emptyset)} u_\infty (y)\Big)+h_1(\emptyset),
\end{equation}
and 
\begin{equation} \label{ig.infty.v.22}
v_\infty (\emptyset) = 
\Big(\frac{1}{m}\sum_{y\in S^1 (\emptyset)} v_\infty (y)\Big)+h_{2}(\emptyset).
\end{equation}

Now, concerning the boundary condition of $u_\infty$, we have
 $$
 \lim_{x\rightarrow z\in\partial\T}u_n(x)= f(\psi(z)).
 $$
 Hence, from the monotonicity of the sequence, we obtain
$$
 \liminf_{x\rightarrow z\in\partial\T}u_\infty (x) \geq \lim_{x\rightarrow z\in\partial\T}u_n (x) = f(\psi(z)).
 $$
 Now, since from Section \ref{sect-soluciones}  we know that we can construct a large supersolution for
 $L_1$ with boundary datum $f$ that we call $\overline{u}$ and we have a comparison principle, we obtain 
 $$
 u_n (x) \leq \overline{u} (x) 
 $$
 for every $x\in\T$ and every $n$ and hence we get
 $$
 u_\infty (x) \leq \overline{u} (x).
 $$
 Thanks to this inequality we obtain
 $$
 \limsup_{x\rightarrow z\in\partial\T}u_\infty (x) \leq 
 \lim_{x\rightarrow z\in\partial\T} \overline{u} (x) = f(\psi(z))
 $$
 and hence we conclude that 
 \begin{equation} \label{limit.u}
 \lim_{x\rightarrow z\in\partial\T}u_\infty (x)  = f(\psi(z)).
 \end{equation}
 
 A similar argument shows that
 \begin{equation} \label{limit.v}
 \lim_{x\rightarrow z\in\partial\T}v_\infty (x)  = g(\psi(z)).
 \end{equation}
 
 To conclude we just observe that 
 \eqref{desig.infty.u}, \eqref{desig.infty.u.v}, \eqref{ig.infty.u}, \eqref{ig.infty.u.22} and \eqref{limit.u}
 show that  
 $$
 u_{\infty}=\ol(L_1,h_1,v_{\infty},f).
 $$
 In the same way \eqref{desig.infty.v}, \eqref{desig.infty.u.v}, \eqref{ig.infty.v},
 \eqref{ig.infty.v.22} and \eqref{limit.v}
 imply that  
 $$
 v_{\infty}=\ou(L_2,h_2,u_{\infty},g).
 $$

Finally, we observe that, since we have a subsolution (called $\underline{u}$) for $L_1$ with datum $f$ and a supersolution (called $\overline{v}$) for $L_2$ with datum $g$ and we assumed that $f>g$, we have
$$
\lim_{x\rightarrow z\in\partial\T}\overline{v} (x)  = g(\psi(z)) <
f (\psi(z)) = \lim_{x\rightarrow z\in\partial\T}\underline{u} (x)  .
$$
Therefore we have that
$$
v_\infty (x) \leq \overline{v} (x) < \underline{u} (x) \leq u_\infty (x)
$$
for every $x$ with $|x|$ large enough, say $|x|>C$. This proves that the contact set
$\{x : u_\infty(x) = v_\infty(x) \}$ does not contain nodes in $\{x : |x|>C\}$ and therefore the contact set 
is finite. 
\end{proof}

	\begin{section}{A probabilistic interpretation of the two membranes problem in the tree} \label{sect-games} \label{sect-probabil}

	Recall that in the introduction we mentioned that the 
system \eqref{ED1.general.intro}, that is given by
\begin{equation}
\label{ED1.general.9977}
\left\lbrace
\begin{array}{l}
\displaystyle u(x) \!= \!
\max\Big\{ \beta_1 u(\hat{x})+(1-\beta_1)\Big(\frac{1}{m} \!\! \sum_{y\in S^1 (x)} u(y)\Big)\!+\!h_1(x),
v(x) \Big\} ,\ x \neq \emptyset, \\[14pt]
\displaystyle
u(\emptyset)\!=\!
\max\Big\{ \Big(\frac{1}{m} \!\! \sum_{y\in S^1 (\emptyset)} u(y)\Big)+h_1(\emptyset),
v(\emptyset) \Big\},  \\[14pt]
\displaystyle v(x)\!= \!\min \Big\{\beta_2 v(\hat{x})+(1-\beta_2)\Big(\frac{1}{m} \!\! \sum_{y\in S^1 (x)} v(y)\Big)\!+\! h_2(x),u(x) \Big\},\ x \neq \emptyset,
\\[14pt]
\displaystyle v(\emptyset)\!= \!\min \Big\{\Big(\frac{1}{m} \!\! \sum_{y\in S^1 (\emptyset)} v(y)\Big)+h_2(\emptyset),u(\emptyset) \Big\}, 
\end{array}
\right.
\end{equation}
has a probabilistic interpretation. In this final section we include the details. 
We refer to the books \cite{BRLibro} and \cite{Lewicka}
for more information concerning games and mean value properties. 

The game is a two-player zero-sum game played in two boards (each board
is a copy of the $m-$regular tree) with the following rules: the game starts with a token at some node 
in one of the two trees $(x_0,i)$ with $x_0\in \T$ and $i=1,2$ (we add an index to denote at which board is the position of the game). In the first board the token is moved to
the predecessor of $x$ or to one of the $m$ successors using probabilities $\beta_1$ and $(1-\beta_1)/m$ while in the second board the probabilities are changed to 
$\beta_2$ and $(1-\beta_2)/m$. At the root of the tree the token moves to one of the successors with probability $1/m$.
In the first board we add a running payoff $h_1(x)$
and in the second board we add $h_2(x)$. In addition to these rules for the movements of the token the players have a choice to play at the same board or to change boards. 
If $x_0$ is in the first board then
Player I (who aims to maximize the expected payoff) decides to remain in the same board 
and play one round of the game moving to the predecessor of $x$ or to one of the $m$ successors using probabilities $\beta_1$ and $(1-\beta_1)/m$
and collecting a running payoff $h_1(x)$ or to jump to the other board. 
On the other hand when $x_0$ is in the second board then it is 
Player II (who aims to minimize the expected payoff) who 
decides to remain in the same board 
and play one round of the game with probabilities $\beta_2$ and $1-\beta_2$
and collecting a running payoff $h_2(x)$ or to jump to the first board.

We take a finite level $M$ (large) and we add the rule that the game ends 
when the position arrives to a node at level $M$, $x_\tau$. We also have two 
final payoffs $f$ and $g$. In the first board Player I pays to Player II the amount encoded 
by $f(\psi(x_\tau))$ while in the second board the final payoff is given by $g(\psi (x_\tau))$
plus the amount encoded in the running payoff. That is, the total payoff of one occurrence
of this game is given by 
$$
\begin{array}{l}
\displaystyle
\mbox{total  payoff} : = f(x_{\tau})\chi_{\{j=1\}}(j_{\tau})+g(x_{\tau})\chi_{\{j=2\}}(j_{\tau})
\\[6pt]
\displaystyle \qquad \qquad \qquad \quad -\sum_{k=0}^{\tau -1}\Big(-h_1 (x_k)\chi_{\{j=1\}}(j_{k+1})-h_2 (x_k)\chi_{\{j=2\}}(j_{k+1})\Big).
\end{array}
$$
Notice that the total payoff is the sum of the final payoff (given by $f(x_{\tau})$ or by $g(x_{\tau})$
according to the board at which the position leaves the domain) and the running payoff
that is given by $h_1(x_k)$ and $ h_2(x_k)$ corresponding to the board in which we play at step $k+1$. 

Notice that the successive positions of the token are determined by
the jumping probabilities given by $\beta_i$ (to jump to the predecessor), $(1-\beta_i)/m$
(to jump to one of the successors) $i=1,2$ according to the board at which the game is played
and the strategies of both players (the choice that one of them make at each turn
regarding the possible change of board). We refer to \cite{BRLibro} for more details
and precise definitions of strategies. We will denote by
$S_I$ a strategy for the first player and $S_{II}$ a strategy for the second player. 

Then the value function 
for this game is defined as
$$
w_M (x,i) = \inf_{S_I} \sup_{S_{II}} \mathbb{E}^{(x,i)} (\mbox{final payoff})
=  \sup_{S_{II}} \inf_{S_I} \mathbb{E}^{(x,i)} (\mbox{final payoff}).
$$
Here the $\inf$ and $\sup$ are taken among all possible strategies of the players.
In this definition of the value of the game we penalize games that never end
(both players may choose to change boards at the same node for ever producing a game that never ends). 
The value of the game $w_M (x,i)$ encodes the amount that the players expect to get/pay
playing their best with final payoffs $f$ and $g$ at level $M$.

We have that the pair of functions $(u_M,v_M)$ given by $u_M(x) = w_M (x,1)$
and $v_M (x) = w_M (x,2)$ is a solution to the system \eqref{ED1.general.9977}
in the finite subgraph of the tree composed by nodes of level less than $M$.

Notice that the first equation encodes all the possibilities for the next position of the game in the first board and includes a maximum since the first player has the choice 
to play or change to the second board. 
Similarly, the second equation
takes into account all the possibilities for the game in the second board and includes a minimum since in this case it is the second player who decides to play or to change boards.

Now our goal is to take the limit as $M \to \infty$ in these value functions for this game and
obtain that the limit is the unique solution to our system that 
verifies the boundary conditions 
\begin{equation}
\label{ED2.general.9977}
\left\lbrace
\begin{array}{ll}
\displaystyle \lim_{{x}\rightarrow z\in\partial\T}u(x) = f(\psi(z)) ,  \\[10pt]
\displaystyle \lim_{{x}\rightarrow z\in\partial\T}v(x)=g(\psi(z)). 
\end{array}
\right.
\end{equation}

\begin{theorem} Fix two continuous functions $f,g:[0,1] \to \mathbb{R}$. 
Let $(u_L,v_L)$ be the values of the game in the finite subgraph of the tree with nodes of level less than $M$, that is, $(u_M,v_M)$ is the solution to \eqref{ED1.general.9977}
with conditions $u_M(x) = f (\psi(x))$ and
$v_L(x) = g (\psi(x))$ at nodes of level $L$.
Then $(u_M,v_M)$ converge, along subsequences, as $M\to \infty$ to $(u,v)$ solutions to the two membranes problem, that is, a solution to 
\eqref{ED1.general.9977} with \eqref{ED2.general.9977}
in the whole tree. 
\end{theorem}

\begin{proof}
From the estimates that we have proved in the previous sections for a
solution $(u_M,v_M)$ we know that these functions are uniformly bounded
in $M$. Therefore, we can extract a subsequence $M_j \to \infty$ such that
$$
u_{M_j}(x) \to u (x) \qquad \mbox{and} \qquad v_{M_j} (x) \to v (x),
$$
for every $x \in \T$. 
Passing to the limit in the equations
$$
\left\lbrace
\begin{array}{l}
\displaystyle u_M(x)=
\max\Big\{ \beta_1 u_M(\hat{x})+(1-\beta_1)\Big(\frac{1}{m}\sum_{y\in S^1 (x)} u_M(y)\Big)+h_1(x),
v_M(x) \Big\},  \\[10pt]
\displaystyle 
u_M(\emptyset)=
\max\Big\{ \Big(\frac{1}{m}\sum_{y\in S^1 (\emptyset)} u_M(y)\Big)+h_1(\emptyset),
v_M(\emptyset) \Big\},  \\[10pt]
\displaystyle v_M(x)=\min \Big\{\beta_2 v_M(\hat{x})+(1-\beta_2)\Big(\frac{1}{m}\sum_{y\in S^1 (x)} v_M(y)\Big)+h_2(x),u_M(x) \Big\},
\\[10pt]
\displaystyle v_M(\emptyset)=\min \Big\{\Big(\frac{1}{m}\sum_{y\in S^1 (\emptyset)} v_M(y)\Big)+h_2(\emptyset),u_M(\emptyset) \Big\}, 
\end{array}
\right.
$$
we get that the limit $(u,v)$ solves
$$
\left\lbrace
\begin{array}{l}
\displaystyle u(x)=
\max\Big\{ \beta_1 u(\hat{x})+(1-\beta_1)\Big(\frac{1}{m}\sum_{y\in S^1 (x)} u(y)\Big)+h_1(x),
v(x) \Big\},  \\[10pt]
\displaystyle 
u(\emptyset)=
\max\Big\{ \Big(\frac{1}{m}\sum_{y\in S^1 (\emptyset)} u(y)\Big)+h_1(\emptyset),
v(\emptyset) \Big\},  \\[10pt]
\displaystyle v(x)=\min \Big\{\beta_2 v(\hat{x})+(1-\beta_2)\Big(\frac{1}{m}\sum_{y\in S^1 (x)} v(y)\Big)+h_2(x),u(x) \Big\},
\\[10pt]
\displaystyle v(\emptyset)=\min \Big\{\Big(\frac{1}{m}\sum_{y\in S^1 (\emptyset)} v(y)\Big)+h_2(\emptyset),u(\emptyset) \Big\}, 
\end{array}
\right.
$$
in the whole tree.

On the other hand, since $f$ and $g$ are continuous, 
 given $\eta>0$ there exists $M$ large enough
such that 
$$
|u_M(x) - \max_{I_x} f | = |f(\psi(x)) - \max_{I_x} f | < \eta, $$
and
$$
|v_M(x) - \max_{I_x} g | = |g(\psi(x)) - \max_{I_x} g |< \eta
$$
for every $x$ at level $M$ with $M$ large enough. Therefore, we get
$$
u_M(x) \leq \max_{I_x} f + \eta, \qquad \mbox{and} \qquad v_M(x) \leq \max_{I_x} g + \eta
$$
for every $x$ at level $M$ large. 
Using that $(u_M,v_M)$ converge as $M \to \infty$ to $(u,v)$ we conclude that 
$$
\left\lbrace
\begin{array}{ll}
\displaystyle \limsup_{{x}\rightarrow z\in\partial\T} u(x) \leq f(\psi(z)) + \eta ,  \\[10pt]
\displaystyle \limsup_{{x}\rightarrow z\in\partial\T}v(x) \leq g(\psi(z))+ \eta. 
\end{array}
\right.
$$
A similar argument using that
$$
u_M(x) \geq \min_{I_x} f - \eta, \qquad \mbox{and} \qquad v_M(x) \geq \min_{I_x} g - \eta
$$
for $M$ large gives that
$$
\left\lbrace
\begin{array}{ll}
\displaystyle \liminf_{{x}\rightarrow z\in\partial\T} u(x) \geq f(\psi(z)) - \eta ,  \\[10pt]
\displaystyle \liminf_{{x}\rightarrow z\in\partial\T}v(x) \geq g(\psi(z)) - \eta. 
\end{array}
\right.
$$

Therefore, since $\eta $ is arbitrary, we conclude that
$(u,v)$ satisfies
$$
\left\lbrace
\begin{array}{ll}
\displaystyle \lim_{{x}\rightarrow z\in\partial\T} u(x) = f(\psi(z)) ,  \\[10pt]
\displaystyle \lim_{{x}\rightarrow z\in\partial\T}v(x) = g(\psi(z)), 
\end{array}
\right.
$$
and the proof is completed. 
\end{proof}

	\end{section}

%%%%%%%%%%%%%%%%%%%%%%%%%%%%%%%%%%%%%%%%%%%%%%%%%%%%%%
	%          7. Acknowlegments
	%%%%%%%%%%%%%%%%%%%%%%%%%%%%%%%%%%%%%%%%%%%%%%%%%%%%%%
	{\bf Acknowledgments}
	
	I. Gonzálvez was partially supported by the European Union's Horizon 2020 research and innovation programme under the Marie Sklodowska-Curie grant agreement No.\,777822, and by grants CEX2019-000904-S, PID-2019-110712GB-I00, PID-2020-116949GB-I00, and RED2022-134784-T, by MCIN/AEI (Spain).
	
	A. Miranda and J. D. Rossi were partially supported by 
			CONICET PIP GI No 11220150100036CO
(Argentina), PICT-2018-03183 (Argentina) and UBACyT 20020160100155BA (Argentina).

	%%%%%%%%%%%%%%%%%%%%%%%%%%%%%%%%%%%%%%%%%%%%%%%%%%%%%%
	%          7. REFERENCES SECTION
	%%%%%%%%%%%%%%%%%%%%%%%%%%%%%%%%%%%%%%%%%%%%%%%%%%%%%%

\end{document}